\documentclass[11pt,a4paper]{article}
\title{\bf Littlewood--Paley--Stein inequalities on $\RCD(K,\infty)$ spaces}
\author{Huaiqian Li\footnote{E-mail address: {\color{blue}huaiqianlee@gmail.com}. }\vspace{3mm}\\
{\footnotesize Center for Applied Mathematics, Tianjin University,
 Tianjin 300072, P. R. China}
}
\date{}
\usepackage{amssymb,amsmath,amsfonts,amsthm,color,mathrsfs}

\def\R{\mathbb{R}}
\def\E{\mathbb{E}}
\def\P{\mathbb{P}}
\def\D{\mathbb{D}}
\def\Z{\mathbb{Z}}

\def\d{\textup{d}}
\def\D{\textup{D}}
\def\Z{\textup{Z}}

\def\CD{\textup{CD}}
\def\Ch{\textup{Ch}}

\def\lip{\textup{lip}}
\def\Lip{\textup{Lip}}

\def\RCD{\textup{RCD}}

\def\<{\langle}
\def\>{\rangle}
\def\Proof.{\noindent{\bf Proof. }}

\def\newdot{{\kern.8pt\cdot\kern.8pt}}

\newtheorem{theorem}{Theorem}[section]
\newtheorem{lemma}[theorem]{Lemma}

\newtheorem{definition}[theorem]{Definition}
\theoremstyle{definition}

\begin{document}
\allowdisplaybreaks
\maketitle
\makeatletter % '@' is now a normal "letter" for TeX
\renewcommand\theequation{\thesection.\arabic{equation}}
\@addtoreset{equation}{section}
\makeatother % '@' is restored as a "non-letter" character for TeX

\begin{abstract}
The $L^p$ boundedness on vertical Littlewood--Paley square functions for heat flows on $\RCD(K,\infty)$ spaces with $K\in\R$ is proved. With regards to the proof, for $1<p\leq 2$, Stein's analytical method is applied, while for $2<p<\infty$, the probabilistic approach in the sense of Ba\~{n}uelos--Bogdan--Luks introduced recently is employed.
\end{abstract}



\section{Introduction and main results}\hskip\parindent
There is a vast amount of study on the boundedness of Littlewood--Paley square functions of various types. Here, we are interested in the $L^p$
boundedness of vertical (i.e., derivative w.r.t. the space variable) ones, since geometric information is closely related. In this aspect, on the one hand, in the celebrated work \cite{St1958} due to Stein, the $L^p$ boundedness of the Littlewood--Paley square function is established for $1<p<\infty$ in the Euclidean space by an analytical approach; see also \cite[Theorem 1, Chapter IV]{St1970}. The result was generalized to other smooth and finite-dimension manifolds; for instance, compact Lie groups (see \cite[Chapter II]{Stein70}) and complete Riemannian manifolds (see \cite{Li2006,CDD}), where, for the case when $1<p\leq 2$, the argument depends on the diffusion property of the second order differential operator. On the other hand, in Meyer's celebrated work \cite{Mey1976}, a probabilistic approach is applied to prove
the Littlewood--Paley--Stein inequality with $2\leq p<\infty$ for the vertical Littlewood--Paley square function defined for general symmetric
Markov process by using the ``carr\'{e} du champ'' operator $\Gamma$ (see also \cite{Mey1981}). Based on the concept of Bakry--Emery
$\Gamma_2$, by developing Meyer's probabilistic method, Bakry \cite{Bakry1987} proved the Littlewood--Paley--Stein inequality (for all
$1<p<\infty$) for a diffusion process on the complete Riemannian manifold under the condition that the Bakry--Emery $\Gamma_2$ is bounded from
below. A closely related work is \cite{ShYo}, where Shigekawa and Yoshida studied the symmetric diffusion process in the metric measure space
setting, but in order to run the $\Gamma_2$ calculus, they assumed the existence of a nice algebra, which is usually difficult to obtain in the study of infinite dimensional diffusions, contained in the domain of the infinitesimal generator of the diffusion process. Some other works we would like to mention here are \cite{Bakry1985,Chen1987,Lohou1987,Sh2002,CD2003}.

In the present work, we investigate the $L^p$ boundedness for vertical Littlewood--Paley square functions for all $1<p<\infty$, namely, the
Littlewood--Paley--Stein inequality, in the setting of metric measure spaces.

Let $(X,d,\mu)$ be a metric measure space in the sense that $(X,d)$ is a separable and complete metric space and $\mu$ is a non-negative Borel
measure on $X$ with full support which is finite on bounded sets. For $1\leq p \leq\infty$, as usual, denote the real $L^p$ space by $L^p(M,\mu)$ and its
norm
\begin{equation*}\|f\|_{L^p(X,\mu)}:=\begin{cases}
\left(\int_X |f(x)|^p\, \d\mu(x)\right)^{1/p},\quad &{1\leq p<\infty},\\

\mathop{\mbox{ess-sup}}\limits_{x\in M} |f(x)|,\quad &{p=\infty},
\end{cases}
\end{equation*}
where $\textup{ess-sup}$ is the essential supremum.

Let $K\in\R$. By using the $K$-convexity of the relative entropy, Lott--Villani \cite{LV2009} and Sturm \cite{Sturm2006a,Sturm2006b} introduced
the notion of curvature-dimension condition as a synthetic notion of Ricci curvature lower bound on the metric measure space. Then the
curvature-dimension condition was strengthened by Ambrosio--Gigli--Savar\'{e} (see \cite{AmbrosioGigliSavare2011b}) by demanding for more Riemannian-like structures and excluding the Finsler structure, which we call the Riemannian curvature-dimension condition (abbrev. $\RCD$), denoted by $\RCD(K,\infty)$; see also \cite{gi2012,AmbrosioGigliSavare2012,agmr2015} as well as Section 2 below. Typical examples of $\RCD(K,\infty)$ spaces are  complete weighted Riemannian manifolds with Bakry--Emery $\Gamma_2$ bounded from below, as well as their limit spaces in the measured Gromov--Hausdorff sense, Alexandorv spaces (with curvature bounded from below), and Hilbert spaces endowed with a log-concave Borel probability measure, and so on.

Let $(X,d,\mu)$ be an $\RCD(K,\infty)$ space with $K\in\R$. There exists a natural notion of heat flow, denoted by $(H_t)_{t\geq0}$, which is
associate to the quasi-regular and strongly local symmetric Dirichlet form induced by the Cheeger energy (see Section 2 below). For $\alpha\geq0$, we set $H_t^\alpha=e^{-\alpha t}H_t$. For every $f\in L^1(X,\mu)\cap L^\infty(X,\mu)$, we define the
vertical Littlewood--Paley square function as
$$\mathcal{G}(f)(x)=\Big(\int_0^\infty |\D H_t^\alpha f|_w^2(x)\,\d t\Big)^{1/2},\quad x\in X,$$
where $|\D f|_w$ is the minimum weak upper gradient of $f$ (see Section 2 below).

The main result of this work is contained in the following theorem.
\begin{theorem}\label{main}
Let $\alpha\geq0$ and let $(X,d,\mu)$ be an $\RCD(K,\infty)$ space with $K\in\R$.
Then, there exists a positive constant $C_p$, depending only on $p$, such that
\begin{itemize}
\item[(1)] if $1<p\leq 2$, then
$$\|\mathcal{G}(f)\|_{L^p(X,\mu)}\leq C_p\|f\|_{L^p(X,\mu)},\quad\mbox{for every }f\in L^2(X,\mu)\cap L^p(X,\mu);$$
\item[(2)] if $\alpha\geq\max\{-2K,0\}$ and $p>2$, then
$$\|\mathcal{G}(f)\|_{L^p(X,\mu)}\leq C_p\|f\|_{L^p(X,\mu)},\quad\mbox{for every }f\in L^2(X,\mu)\cap L^p(X,\mu).$$
\end{itemize}
\end{theorem}
We remark that, on the $\RCD(K,\infty)$ space $(X,d,\mu)$ with $K\in\R$, the horizontal (i.e., derivative w.r.t. the time variable) Littlewood--Paley square function, defined as
$$g_{k}(f)(x)=\Big(\int_0^\infty \Big|t^{2k}\frac{\partial^k}{\partial t^k}H_t^\alpha f(x)\Big|^2\,\frac{dt}{t} \Big)^{1/2},\quad x\in X,$$
for every $f\in L^p(X,\mu)$ and all $k=1,2,\cdots$, is bounded in $L^p(X,\mu)$ for all $1<p<\infty$ by the result in \cite{CRW}. However, we are not allowed to apply \cite[COROLLARY 1]{Stein70} on Page 120 directly to derive the $L^p$ boundedness of $g_k$ due to the lack of stochastic completeness for $(H^\alpha_t)_{t>0}$ in general; indeed, when $\alpha>0$, $H_t^\alpha1<1$ for every $t>0$ (see Section 2). Moreover, inspired by \cite[Remark 1.3 (ii)]{CDD}, for every $f\in L^1(X,\mu)\cap L^\infty(X,\mu)$, if we define the square function as
$$\widetilde{\mathcal{G}}(f)(x)=\Big(\int_0^\infty t\big|\D e^{-t\sqrt{\alpha-\Delta}}f)\big|^2_w(x)\,\d t\Big),\quad x\in X,$$
then $\widetilde{\mathcal{G}}(f)$ is pointwise dominated by $\mathcal{G}(f)$. As a result, the $L^p$ boundedness of $\mathcal{G}$ implies the $L^p$ boundedness of $\widetilde{\mathcal{G}}$.

In Section 2, we recall some basic notions and known results following the works due to Ambrosio--Gigli--Savar\'{e}.  In the subsequent sections, we present the proof of Theorem \ref{main}, which is divided into two parts: Section 3 on the $L^p$ boundedness for $1<p\leq2$, and Section 4 on the $L^p$ boundedness for $2< p<\infty$.

\section{Preliminaries}\hskip\parindent
Let $(X,d)$ be a metric space. Denote $\Lip(X,d)$ the class of all Lipschitz continuous functions on $(X,d)$. Given a function $f:X\rightarrow\R$, define its local Lipschitz constant $\lip(f): X\rightarrow[0,\infty]$ by
\begin{equation*}
\lip(f)(x)=
\begin{cases}
\limsup\limits_{y\rightarrow x}\frac{|f(y)-f(x)|}{d(x,y)},\quad &{\mbox{if }x\mbox{ is not isolated}},\\
0,\quad &{\hbox{otherwise}}.
\end{cases}
\end{equation*}

Let $(X,d)$ be a separable and complete metric space, and let $\mu$ be a non-negative Borel measure on $X$ which has full support and is finite on bounded subsets of $X$. The triple $(X,d,\mu)$ is called a metric measure space.

The Sobolev space $W^{1,2}(X):=W^{1,2}(X,d,\mu)$ over the metric measure space $(X,d,\mu)$ is defined as the class of all functions $f\in L^2(X,\mu)$ for which there exists a sequence of functions $(f_n)_{n\geq1}\subseteq L^2(X,\mu)\cap\Lip(X,d)$ such that $f_n\rightarrow f$ and $\sup_{n\geq1}\int_X\lip(f_n)^2\,\d\mu<\infty$. For every $f\in W^{1,2}(X)$, denote its minimal weak upper gradient by $|\D f|_w$. We refer to \cite{AGS2014,gi2012} for the existence and properties on the minimal weak upper gradient. It turns out that $W^{1,2}(X)$ equipped with the norm $\|\cdot\|_{W^{1,2}(X)}$, defined by
$$\|f\|^2_{W^{1,2}(X)}:=\|f\|_{L^2(X,\mu)}^2+\||\D f|_w\|_{L^2(X,\mu)}^2,$$
is a Banach space, but not a Hilbert space in general.

In order to rule out Finsler structures, the Riemannian curvature-dimension condition (abbrev. $\RCD$) is introduced in \cite{AmbrosioGigliSavare2011b} (with the reference measure a probability measure) and then in \cite{agmr2015} (with the
reference measure as the same as $\mu$), which is a strengthening of the curvature-dimension condition $\CD(K,\infty)$ in the sense of Lott--Sturm--Villani (see \cite{Sturm2006a,LV2009}) by further requiring the Banach space $(W^{1,2}(X),\|\cdot\|_{W^{1,2}(X)})$ to be a Hilbert space.
\begin{definition}
Let $K\in \R$ and let $(X,d,\mu)$ be a metric measure space. We say that $(X,d,\mu)$ is an $\RCD(K,\infty)$ space (or $(X,d,\mu)$ satisfies the
$\RCD(K,\infty)$ condition) if it satisfies $\CD(K,\infty)$ and $(W^{1,2}(X),\|\cdot\|_{W^{1,2}(X)})$ is a Hilbert space.
\end{definition}

From now on, let $(X,d,\mu)$ be an $\RCD(K,\infty)$ space with $K\in\R$. Then, there is a natural quadratic form induced by the Cheeger energy. Recall that the Cheeger energy functional $\Ch: L^2(X,\mu)\rightarrow[0,\infty]$ is defined by
$$\Ch(f)=\frac{1}{2}\int_X|\D f|_w^2\,\d\mu,\quad f\in W^{1,2}(X).$$
By polarization, we can define a bilinear map $\Gamma: W^{1,2}(X)\times W^{1,2}(X)\rightarrow\R$ by
$$\Gamma(f,g)=\frac{1}{4}\big(|\D(f+g)|_w^2-|\D(f-g)|_w^2\big),\quad\mbox{for every }f,g\in W^{1,2}(X).$$
Hence, $\Ch$ induces a quadratic form $\mathcal{E}:W^{1,2}(X)\times W^{1,2}(X)\rightarrow\R$ such that, for any $f,g\in W^{1,2}(X)$,
$\mathcal{E}(f,f)=2\Ch(f)$ and
$$\mathcal{E}(f,g)=\int_X\Gamma(f,g)\,\d\mu.$$
Indeed, $(\mathcal{E}, W^{1,2}(X))$ is a strongly local, quasi-regular symmetric Dirichlet form (see \cite[Section
6.2]{AmbrosioGigliSavare2011b} and \cite[Section 7.2]{agmr2015}). We denote its $L^2$-generator by $\Delta$ with domain $\mathcal{D}(\Delta)$ belong to $L^2(X,\mu)$, which
is a self-adjoint and non-positive definite linear operator such that, for every $f\in\mathcal{D}(\Delta)$,
$$\mathcal{E}(f,\psi)=-\int_X \psi\Delta f\,\d\mu,\quad\mbox{for every }\psi\in W^{1,2}(X).$$
From the general theory of Dirichlet forms (see e.g. \cite{BouleauHirsch,MaRo1992,FOT2011} for more details),
$\mathcal{D}(\Delta)$ is dense in $W^{1,2}(X)$ as well as in $L^2(X,\mu)$. The heat flow uniquely associated to $(\mathcal{E},
W^{1,2}(X))$ is denoted by $(H_t)_{t\geq0}$, namely, $H_t=e^{t\Delta}$ in the sense of functional analysis, and the heat kernel by $(h_t)_{t\geq0}$ (see \cite[Section 6.1]{AmbrosioGigliSavare2011b} and \cite[Section 7.1]{agmr2015} for the existence of the later), which enjoy nice properties (see \cite[Section 6.1]{AmbrosioGigliSavare2011b},
\cite{Savare2014}, \cite[Section 7.1]{agmr2015}, \cite{AmbrosioGigliSavare2012} and \cite{Li2016}); for instance, the mass preserving property, i.e.,
\begin{equation}\label{mass-pres}\int_X H_tf\,\d\mu=\int_X f\,\d\mu,\quad\mbox{for every }f\in L^1(X,\mu)\cap L^2(X,\mu),\end{equation}
 the density of
$\cup_{t>0}H_t(L^\infty(X,\mu))$ in $\mathcal{D}(\Delta)$, and the stochastic completeness, i.e., for any $t>0$, $H_t1(x)=1$ for $\mu$-a.e. $x\in X.$

Note that the doubling property and the weak local Poincar\'{e} inequality of type $(1,1)$ holds on $\RCD(K,\infty)$ spaces with $K\geq0$. However, the both may fail on $\RCD(K,\infty)$ spaces with $K<0$.

For $\alpha\geq0$, let $(H_t^\alpha)_{t\geq0}$ be the semigroup generated by $(\Delta-\alpha,\mathcal{D}(\Delta))$ in $L^2(X,\mu)$. Then $H_t^\alpha=e^{-\alpha t}H_t$, $t\geq0$, which is also a strongly continuous and $L^p(X,\mu)$-contraction semigroup for all $p\in[1,\infty]$. However, for $\alpha>0$, $H_t^\alpha 1<1$ $\mu$-a.e. for all $t>0$.

Now for $f\in L^1(X,\mu)\cap L^\infty(X,\mu)$, we define the vertical Littlewood--Paley square function by
$$\mathcal{G}(f)(x)=\Big(\int_0^\infty |\D H_t^\alpha f|_w^2(x)\,\d t\Big)^{1/2},\quad x\in X.$$
It is easy to know that if $\alpha\geq0$, then
\begin{eqnarray}\label{L2-bound}
\|\mathcal{G}(f)\|_{L^2(X,\mu)}\leq\frac{\sqrt{2}}{2}\|f\|_{L^2(X,\mu)},\quad\mbox{for every }f\in L^2(X,\mu).
\end{eqnarray}
Indeed, letting $\{E_\lambda: 0\leq\lambda<\infty\}$ be the spectral family associated with $-\Delta$, we have that
\begin{eqnarray*}
&&\int_M|\mathcal{G}(f)|^2\,d\mu=\int_M\int_0^\infty|\D H_t^\alpha f|_w^2\,\d t \d\mu=\int_0^\infty\Big(\int_M -\Delta(H_t^\alpha f) H_t^\alpha f\,\d\mu\Big)\,\d t\\
&=&\int_0^\infty\int_0^\infty \lambda e^{-2t(\alpha+\lambda)}\,\d\langle E_\lambda f,E_\lambda f\rangle\,\d t=\int_{0^+}^\infty\int_0^\infty \lambda
e^{-2t(\alpha+\lambda)}\,\d t\,\d\langle E_\lambda f,E_\lambda f\rangle\\
&\leq&\frac{1}{2}\int_{0^+}^\infty\d\langle E_\lambda f,E_\lambda f\rangle\leq\frac{1}{2}\|f\|^2_{L^2(X,\mu)}.
\end{eqnarray*}

\section{$L^p$ boundedness for $1<p\leq2$}\hskip\parindent
We first borrow a maximal inequality on the heat flows $(H_t)_{t>0}$ from \cite[Page 73]{Stein70} (see also \cite[Lemma 2.7]{LiWang16}); in fact, the explicit constant for $1<p<\infty$ in the present form is from \cite[Theorem 3.3]{Shig2004}, where the method of proof still works in the present setting.
\begin{lemma}\label{max-ergo-ineq}
Let $1<p\leq\infty$. For every $f\in L^p(X,\mu)$,
$$\big\|\sup_{t>0}|H_tf|\big\|_{L^p(X,\mu)}\leq \frac{p}{p-1}\|f\|_{L^p(X,\mu)},$$
where $p/(p-1):=1$ if $p=\infty$.
\end{lemma}

The main result in this section is presented in the following theorem which is exact Theorem \ref{main} (1). The idea of proof is essentially from Stein \cite[Chapter II]{Stein70} in the the setting of compact Lie groups. We should mention that Stein's method has been non-trivially generalised to the setting of non-local Dirichlet forms under some mild conditions in \cite{LiWang16} very recently. For notational convenience, we set
$$\mathbb{V}^\infty(X):=\{g\in W^{1,2}(X)\cap L^\infty(X,\mu):\, |\D H_tg|_w\in L^\infty(X,\mu)\}.$$
\begin{theorem}\label{p<2}
Let $\alpha\geq0$ and $1<p\leq2$. Suppose that $(X,d,\mu)$ is an $\RCD(K,\infty)$ space with $K\in\R$.
Then, there exists a positive constant $C_p$, depending only on $p$, such that
\begin{equation}\label{thm3.2-1}
\|\mathcal{G}(f)\|_{L^p(X,\mu)}\leq C_p\|f\|_{L^p(X,\mu)},\quad\mbox{for every }f\in L^2(X,\mu)\cap L^p(X,\mu).
\end{equation}
\end{theorem}
\begin{proof}
We only need to consider the case when $1<p<2$ since \eqref{L2-bound}. We may reduce the problem to the case when $f$ is non-negative. Indeed, if $f$ is signed then $f=f^+-f^-$ and
$|\D H^\alpha_tf|_w\leq |\D H^\alpha_t(f^+)|_w + |\D H^\alpha_t(f^-)|_w$, where $f^+:=\max\{f, 0\}$ and $f^-:=\{-f,0\}$; hence,
$$\|\mathcal{G}(f)\|^p_{L^p(X,\mu)}\leq 2^{p-1}\big(\|\mathcal{G}(f^+)\|^p_{L^p(X,\mu)}+\|\mathcal{G}(f^-)\|^p_{L^p(X,\mu)}\big).$$

Let $0\leq f\in L^1(X,\mu)\cap L^\infty(X,\mu)$. By the the property of the heat flow (see e.g. \cite[Section 6]{AmbrosioGigliSavare2011b} and  \cite[Section 7]{agmr2015}), we see that $0\leq H_tf\in \mathcal{D}(\Delta)\cap\Lip(X,d)\cap L^\infty(X,\mu)$ for every $t>0$.
By \cite[Corollary 3.2]{AmbrosioGigliSavare2012}, we have that
$$|\D H_tf|^2_w\leq\frac{2K}{e^{2Kt}-1} H_t(f^2),\quad \mu\mbox{-a.e. in }X,$$
which implies that $|\D H_tf|_w\in L^\infty(X,\mu)$ for every $t>0$. Thus, $0\leq H^\alpha_tf\in \mathcal{D}(\Delta)\cap
\mathbb{V}^\infty(X)\cap \Lip(X,d)$ for every $t>0$.

Let $\epsilon>0$. Define a function $\theta_\epsilon: \R_+\rightarrow\R_+$ such that $\theta_\epsilon(s)=(s+\epsilon)^p-\epsilon^p$. Then, it
is immediate to see that $\theta_\epsilon(0)=0$ and $\theta_\epsilon\in C^2(\R_+,\R_+)$. From \cite[Corollary 6.1.4, Chapter I]{BouleauHirsch},
we deduce that $\theta_\epsilon(H^\alpha_tf)\in \mathcal{D}(\Delta)\cap \mathbb{V}^\infty(X)\cap \Lip(X,d)$ for every $t>0$. (Note that here we do not need the boundedness on the derivatives of $s\mapsto\theta_\epsilon(s)$, since $\theta_\epsilon$ is composed with bounded functions.) Hence, by the chain rule (see e.g. \cite{gi2012}),
\begin{eqnarray*}
&&\Big(\frac{\partial}{\partial t}-\Delta\Big)\theta_\epsilon(H^\alpha_tf)\\
&=&\theta'_\epsilon(H^\alpha_tf)\frac{\partial}{\partial t}H_t^\alpha f-
\big[\theta'_\epsilon(H^\alpha_tf)\Delta H^\alpha_tf+\theta''_\epsilon(H^\alpha_tf)|\D H^\alpha_t f|_w^2\big]\\
&=&-\alpha\theta'_\epsilon(H^\alpha_tf) H^\alpha_tf - \theta''_\epsilon(H^\alpha_tf)|\D H^\alpha_tf|^2_w.
\end{eqnarray*}
Since $\theta_\epsilon''(s)=p(p-1)(s+\epsilon)^{p-2}$, which is positive for any $s\in\R_+$, we have that $$\big(\frac{\partial}{\partial t}-\Delta\big)\theta_\epsilon(H^\alpha_tf)+\alpha\theta'_\epsilon(H^\alpha_tf) H^\alpha_tf\leq0,$$ and
$$|\D H^\alpha_tf|^2_w=\frac{1}{\theta''_\epsilon(H^\alpha_tf)}\Big[\Big(\Delta-\frac{\partial}{\partial t}\Big)\theta_\epsilon(H^\alpha_tf)-\alpha \theta'_\epsilon(H^\alpha_tf)H^\alpha_tf\Big].$$
Then
\begin{eqnarray*}
\mathcal{G}(f)(x)^2&=&\int_0^\infty |\D H^\alpha_tf|^2_w(x)\,\d t\\
&=&\int_0^\infty \frac{1}{\theta''_\epsilon(H^\alpha_tf(x))}\Big[\Big(\Delta-\frac{\partial}{\partial t}\Big)\theta_\epsilon(H^\alpha_tf(x))-\alpha \theta'_\epsilon(H^\alpha_tf(x))H^\alpha_tf(x)\Big]\,\d t\\
&\leq& \sup_{t>0}\Big(\frac{1}{\theta''_\epsilon(H^\alpha_tf(x))}\Big)I_\epsilon(x),
\end{eqnarray*}
where
$$I_\epsilon(x):=\int_0^\infty \Big[\Big(\Delta-\frac{\partial}{\partial t}\Big)\theta_\epsilon(H^\alpha_tf(x))-\alpha \theta'_\epsilon(H^\alpha_tf(x))H^\alpha_tf(x)\Big]\,\d t\geq0.$$

By H\"{o}lder's inequality, we derive that
\begin{eqnarray*}
\int_X \mathcal{G}(f)(x)^p\,\d\mu(x)&\leq& \int_X \sup_{t>0}\Big(\frac{1}{\theta''_\epsilon(H^\alpha_tf(x))}\Big)^{p/2}I_\epsilon(x)^{p/2}\,\d\mu(x)\\
&\leq&\Big(\int_X
\sup_{t>0}\Big(\frac{1}{\theta''_\epsilon(H^\alpha_tf)}\Big)^{p/(2-p)}\,\d\mu\Big)^{1-p/2}\Big(\int_XI_\epsilon(x)\,\d\mu(x)\Big)^{p/2}.
\end{eqnarray*}
Since  $\theta'(s)=p(s+\epsilon)^{p-1}>0$ for any $s\in\R_+$, the term
\begin{eqnarray*}
\int_XI_\epsilon(x)\,\d\mu(x)&=&\int_X\int_0^\infty \Delta \theta_\epsilon(H^\alpha_tf)\,\d t\d\mu
-\int_X\int_0^\infty \frac{\partial}{\partial t}\theta_\epsilon(H^\alpha_tf)\,\d t\d\mu\\
&&-\alpha\int_0^\infty \theta_\epsilon'(H^\alpha_tf)H^\alpha_tf\,\d t\\
&\leq&\int_X\int_0^\infty \Delta \theta_\epsilon(H^\alpha_tf)\,\d t\d\mu+\int_X \theta_\epsilon(f)\,\d\mu\\
&=&\int_X \theta_\epsilon(f)\,\d\mu,
\end{eqnarray*}
where we also used Fubini's theorem and the fact that $\int_X \Delta \theta_\epsilon(H^\alpha_tf)\, \d\mu=0$ due to the mass preserving property \eqref{mass-pres} in the last inequality.
Note that $\int_X \theta_\epsilon(f)\,\d\mu<\infty$, which comes from the fact that $\theta_\epsilon(s)\leq ps(s+\epsilon)^{p-1}$ for any $s\in\R_+$
and $0\leq f\in L^1(X,\mu)\cap L^\infty(X,\mu)$. Hence, letting $\epsilon\rightarrow0^+$, by the monotone convergence theorem and the dominated
convergence theorem, we immediately arrive at
\begin{eqnarray*}
\int_X \mathcal{G}(f)^p\,\d\mu&\leq&\lim_{\epsilon\rightarrow0^+}
\Big(\int_X\sup_{t>0}\Big(\frac{1}{\theta''_\epsilon(H^\alpha_tf)}\Big)^{p/(2-p)}\,\d\mu\Big)^{1-p/2}\Big(\int_X
\theta_\epsilon(f)\,\d\mu\Big)^{p/2}\\
&=&[p(p-1)]^{-p/(2-p)} \Big(\int_X\sup_{t>0}(H^\alpha_tf)^p\,\d\mu\Big)^{(2-p)/2}\|f\|_{L^p(X,\mu)}^{p^2/2}\\
&\leq& C_p\|f\|_{L^p(X,\mu)}^{p},
\end{eqnarray*}
where Lemma \ref{max-ergo-ineq} is applied in the last inequality. Thus, we prove \eqref{thm3.2-1} for $f\in L^1(X,\mu)\cap L^\infty(X,\mu)$.

For $f\in  L^2(X,\mu)\cap L^p(X,\mu)$, we may choose a sequence $(f_n)_{n\geq1}$ from $L^1(X,\mu)\cap L^\infty(X,\mu)$ such that $f_n\rightarrow f$ in $L^2(X,\mu)$ as $n\rightarrow\infty$. Then, for every $t>0$, $H^\alpha_tf_n\rightarrow H^\alpha_tf$ in $W^{1,2}(X)$ as $n\rightarrow\infty$. Hence, up to subsequence, for every $t>0$, $|\D H^\alpha_tf_n|_w\rightarrow |\D H^\alpha_tf|_w$ $\mu$-a.e. as $n\rightarrow\infty$. Therefore, by Fatou's lemma, we complete the proof.
\end{proof}

\section{$L^p$ boundedness for $2<p<\infty$}\hskip\parindent
Let $K\in\R$ and $(X,d,\mu)$ be an $\RCD(K,\infty)$ space. Since $(\mathcal{E}, W^{1,2}(X))$ is a stochastically complete, strongly local
 and quasi-regular Dirichlet form, from \cite[Theorem 1.11, Chapter V]{MaRo1992},  there exists a unique diffusion process $\{(\Z_t)_{t\geq0},
(\P_x)_{x\in X}\}$ defined on a probability space such that $H_tf(x)=\E_x[f(\Z_t)]$ for any $t\geq0$, $x\in X$ and non-negative Borel function $f$ on $X$, and
$$\P_x\big(t\mapsto \Z_t\mbox{ is continuous for }t\in[0,\infty)\big)=1,\quad\mbox{for every }x\in X,$$
where $\E_x$ denotes the expectation with respect to $\P_x$. See also \cite[Theorem 6.8]{AmbrosioGigliSavare2011b} and \cite[Theorem 7.5]{agmr2015}.

Now we adapt the probabilistic approach from \cite{BBL2016},
where the $L^p$ boundedness of the Littlewood--Paley square function for L\'{e}vy processes in the Euclidean space is studied.

Fix an arbitrary number $T>0$ and let $f\in L^1(X,\mu)\cap L^\infty(X,\mu)$. Define a stochastic processes $\mathcal{M}_f=(\mathcal{M}_f(t))_{0\leq t\leq T}$ by
$$\mathcal{M}_f(t)=H_{T-t}f(\Z_t)-H_Tf(\Z_0),\quad 0\leq t\leq T.$$
Denote the natural filtration of the process $(\Z_t)_{t\geq0}$ by $(\mathcal{F}_t)_{t\geq0}$. Then Lemma \ref{martingale} below shows that
$(\mathcal{M}(f)_t,\mathcal{F}_t)_{0\leq t\leq T}$ is a martingale.

The next result is not new at lest in the smoothing setting since one can apply the It\^{o} formula.  For our metric measure space setting, we refer to \cite{Li2017} for the same proof; however, instead of $\RCD(K,\infty)$, the stronger Riemannian curvature-dimension condition $\RCD^\ast(K,N)$ with finite $N$ is assumed there. See \cite{gi2012,eks2013} for more details on the $\RCD^\ast(K,N)$ condition.
\begin{lemma}\label{martingale}
Let $(X,d,\mu)$ be an $\RCD(K,\infty)$ space with $K\in\R$. Then $(\mathcal{M}_f(t),\mathcal{F}_t)_{0\leq t\leq T}$ defined above is a
continuous martingale, and moreover, for any $t\in[0,T]$, its quadratic variations is
$$\langle\mathcal{M}_f\rangle(t)=\int_0^t|\D H_{T-s}f|_w^2(\Z_s)\,\d s.$$
\end{lemma}

For $x\in X$ and $f\in L^1(X,\mu)\cap L^\infty(X,\mu)$, define
$$\mathcal{G}_\star(f)(x)=\Big(\int_0^\infty H_t^\alpha\big(|\D H_t^\alpha f|_w^2\big)(x)\,\d t\Big)^{1/2},$$
$$\mathcal{G}_{\star, T}(f)(x)=\Big(\int_0^T H_t^\alpha\big(|\D H_t^\alpha f|_w^2\big)(x)\,\d t\Big)^{1/2},$$
and
$$\mathcal{H}_{\star, T}(f)(x)=\Big(\int_0^T H_t\big(|\D H_t f|_w^2\big)(x)\,\d t\Big)^{1/2}.$$
Obviously, $\mathcal{G}_{\star, T}(f)(x)$ increases to $\mathcal{G}_\star(f)(x)$ as $T\rightarrow\infty$, and $\mathcal{G}_{\star, T}(f)(x)\leq\mathcal{H}_{\star, T}(f)(x)$ since $\alpha\geq0$. And we have the following lemma.
\begin{lemma}\label{G-G_star-lemma}
Let $(X,d,\mu)$ be an $\RCD(K,\infty)$ space with $K\in\R$. Suppose that $\alpha\geq\max\{-2K,0\}$. Then, for every $f\in L^1(X,\mu)\cap L^\infty(X,\mu)$ and $x\in X$,
\begin{eqnarray}\label{G-G_star}
\mathcal{G}(f)(x)\leq\sqrt{2}\mathcal{G}_\star(f)(x).
\end{eqnarray}
\end{lemma}
\begin{proof}
Since for every $t>0$, $|\D H_tf|_w^2\leq e^{-2Kt}H_t(|\D f|_w^2)$ for every $f\in W^{1,2}(X)$ (see \cite[COROLLARY 4.18]{AmbrosioGigliSavare2012} for instance) and $H_t^\alpha f\in W^{1,2}(X)$ for every $f\in L^1(X,\mu)\cap L^\infty(X,\mu)$, it is easy to see that, if $\alpha\geq\max\{-2K,0\}$, then for every $x\in X$,
\begin{eqnarray*}
\mathcal{G}(f)^2(x)&=&\int_0^\infty|\D H_t^\alpha f|_w^2\,\d t=\int_0^\infty|\D H_{t/2}^\alpha(H_{t/2}^\alpha f)|_w^2\,\d t\\
&\leq&\int_0^\infty e^{-(\alpha/2+K)t}H_{t/2}^\alpha(|\D H_{t/2}^\alpha f|^2_w)\,\d t\leq2\mathcal{G}_\star(f)^2(x),
\end{eqnarray*}
which is right \eqref{G-G_star}.
\end{proof}

The next lemma contains the key point we need, which expresses the square function $\mathcal{H}_{\star,T}(f)$ as a
conditional expectation of the quadratic variation of the martingale $\mathcal{M}_f$.  The idea is introduced in the paper \cite{BBL2016} for L\'{e}vy processes recently.  The proof is the same to the one in \cite[Lemma 3.5]{Li2017} (where the $\RCD^\ast(K,N)$ with finite $N$ is assumed). Recall that $h_t$ is the heat kernel corresponding to $H_t$.
\begin{lemma}\label{G_T}
Let $(X,d,\mu)$ be an $\RCD(K,\infty)$ space with $K\in\R$ and $x\in X$.
Then
\begin{eqnarray*}\mathcal{H}_{\star,T}(f)(x)=\Big\{\int_X \mathbb{E}_z\Big(\int_0^T|\D H_{T-s}f|_w^2(\Z_s)\,\d s\Big|
\Z_T=x\Big)h_T(x,z)\,\d\mu(z)\Big\}^{1/2}.
\end{eqnarray*}
\end{lemma}

The main result of this section is the next theorem which is the second part of Theorem \ref{main}.
\begin{theorem}\label{p>2}
 Let $2\leq p<\infty$ and let $(X,d,\mu)$ be an $\RCD(K,\infty)$ space with $K\in\R$. Suppose that $\alpha\geq\max\{-2K,0\}$. Then, there exits a positive constant $C_p$ depending only on $p$ such that
$$\|\mathcal{G}(f)\|_{L^p(X,\mu)}\leq C_p\|f\|_{L^p(X,\mu)},\quad\mbox{for every }f\in L^2(X,\mu)\cap L^p(X,\mu).$$
\end{theorem}
\begin{proof} Let $f\in L^1(X,\mu)\cap L^\infty(X,\mu)$. For $2\leq p<\infty$, by Lemma \ref{martingale} and Lemma \ref{G_T},
\begin{eqnarray*}
\int_X \mathcal{H}_{\star,T}(f)(x)^p\,\d\mu(x)&=&\int_X\Big(\int_X\mathbb{E}_y\big[\langle\mathcal{M}_f\rangle(T)\big|
\Z_T=x\big]h_T(x,y)\,\d\mu(y)\Big)^{p/2}\d\mu(x)\\
&\leq&\int_X\int_X\mathbb{E}_y\big[\langle\mathcal{M}_f\rangle^{p/2}(T)\big| \Z_T=x\big]h_T(x,y)\,\d\mu(y)\d\mu(x)\\
&=&\int_X\mathbb{E}_y\big[\langle\mathcal{M}_f\rangle^{p/2}(T)\big]\,\d\mu(x),
\end{eqnarray*}
where we applied Jensen's inequality. In what follows, the positive constant $C_p$, depending only on $p$, may vary from line to line. By the Burkholder--Davis--Gundy
inequality (see e.g. \cite[Theorem 3.1]{Shig2004}), we have
\begin{eqnarray*}
\int_X \mathcal{H}_{\star,T}(f)(x)^p\,\d\mu(x)&\leq& C_p\int_X \mathbb{E}_y|\mathcal{M}_f(T)|^p\,\d\mu(y)\\
&\leq&  C_p\int_X \big(\mathbb{E}_y|f(\Z_T)|^p + |H_Tf(y)|^p\big)\,\d\mu(y)\\
&=&  C_p\int_X \big( H_T|f|^p(y)+|H_Tf(y)|^p\big)\,\d\mu(y)\\
&\leq& C_p\|f\|^p_{L^p(X,\mu)},
\end{eqnarray*}
where we used the stochastic completeness in the last inequality. Thus, by \eqref{G-G_star} and the monotone convergence theorem,
\begin{eqnarray*}
\int_X \mathcal{G}(f)(x)^p\,\d\mu(x)&\leq& 2^{p/2}\int_X\mathcal{G}_{\star}(f)(x)^p\,\d\mu(x)\\
&=& 2^{p/2}\lim_{T\rightarrow\infty} \int_X\mathcal{G}_{\star,T}(f)(x)^p\,\d\mu(x)\\
&\leq&2^{p/2}\lim_{T\rightarrow\infty} \int_X\mathcal{H}_{\star,T}(f)(x)^p\,\d\mu(x)\\
&\leq&  C_p\|f\|^p_{L^p(X,\mu)}.
\end{eqnarray*}
For general $f\in L^2(X,\mu)\cap L^p(X,\mu)$, by approximation and Fatou's lemma, we complete the proof.
\end{proof}

\subsection*{Acknowledgment}\hskip\parindent
The author would like to thank Dr. Bang-Xian Han for a nice discussion on Theorem \ref{p<2}. The author would also like to acknowledge the financial support from the National Natural Science Foundation of China (Nos. 11401403, 11571347 and 11831014).

\end{document}